\documentclass[11pt,oneside]{amsart} 
\usepackage{amsmath, amssymb, verbatim, amscd, amsthm, mathrsfs, mathtools}
\usepackage[driverfallback=dvipdfm]{hyperref}
\usepackage{color, graphicx, shortvrb}
\usepackage{enumerate}
\usepackage[all]{xy}

\numberwithin{equation}{section}
\theoremstyle{definition} 
\newtheorem{defn}{Definition}[section]

\theoremstyle{theorem}
\newtheorem{thm}[defn]{Theorem}

\newtheorem{lem}[defn]{Lemma}
\newtheorem{prop}[defn]{Proposition} 
\newtheorem{cor}[defn]{Corollary}

\theoremstyle{remark}
\newtheorem{rem}[defn]{Remark}
\newtheorem*{rem*}{Remark}

\theoremstyle{theorem}
\newcounter{proofcount}
\AtBeginEnvironment{proof}{\stepcounter{proofcount}}
\newtheorem{step}{Step}
\makeatletter
\@addtoreset{step}{proofcount}
\makeatother

\theoremstyle{remark}
\makeatletter
\newtheorem*{sproof/}{Proof of Step \rev@sproofmark}
\newenvironment{sproof}[1][\@nil]
  {\def\@tmp{#1}%
   \ifx\@tmp\@nnil
       \def\rev@sproofmark{\thestep}% no optional argument: take last claim
    \else
       \let\rev@sproofmark\@tmp% take the argument
    \fi
   \pushQED{\qed}\begin{sproof/}}
  {\popQED\end{sproof/}}
\makeatother

\newcommand{\q}{\quad} 
 
\newcommand{\qbox}[1]{\q \mbox{#1} \q}

\newcommand{\C}{\mathbb C} 
\newcommand{\N}{\mathbb N} 
\newcommand{\R}{\mathbb R}

\def\a{\alpha}

\def\Gam{\Gamma}

\def\d{\delta}

\newcommand{\set}[1]{\left\{ #1  \right\}}
\newcommand{\norm}[1]{\left\lVert #1 \right\rVert}
\newcommand{\ab}[1]{\left\lvert #1 \right\rvert}
\newcommand{\inn}[1]{\left\langle #1 \right\rangle}
\newcommand{\ext}[1]{\mathrm{ext}(#1)}

\newcommand{\cx}{C_0(X, \R)}
\newcommand{\cy}{C_0(Y, \R)}
\newcommand{\cl}{C_0(L, \R)}
\newcommand{\sx}{S(\cx)}
\newcommand{\sy}{S(\cy)}
\newcommand{\sol}{S(\cl)}

\usepackage[normalem]{ulem}
\usepackage{marginnote}
\usepackage[top=30truemm,bottom=25truemm,left=35truemm,right=35truemm]{geometry}

\begin{document}

\title[Phase-isometries between the unit spheres]{On phase-isometries between the unit spheres of the Banach space of continuous real-valued functions}

\author[Y.~Enami]{Yuta Enami}
\address[Yuta Enami]{Graduate School of Science and Technology,
Niigata University, Niigata 950-2181, Japan}
\email{f21j008j@mail.cc.niigata-u.ac.jp}

\author[I.~Matsuzaki]{Izuho Matsuzaki}
\address[Izuho Matsuzaki]{Graduate School of Science and Technology,
Niigata University, Niigata 950-2181, Japan}
\email{matsuzaki@m.sc.niigata-u.ac.jp}


\date{}

\begin{abstract}
    For a locally compact Hausdorff space $L$, 
    we denote by $C_0(L,\mathbb{R})$ 
    the Banach space of all continuous
    real-valued functions on $L$ 
    vanishing at infinity, endowed with 
    the supremum norm. 
    In this paper, we prove that every surjective phase-isometry $T\colon S(C_0(X,\mathbb{R}))\to S(C_0(Y,\mathbb{R}))$ between 
    the unit spheres of $C_0(X,\mathbb{R})$ and $C_0(Y,\mathbb{R})$ is 
    a variant of a 
    weighted composition operator in the following sense:
    there exist a function 
    $\varepsilon\colon S(C_0(X,\mathbb{R}))\to\set{-1,1}$,
    a continuous function $\a\colon Y\to \set{-1,1}$ 
    and a homeomorphism $\sigma\colon Y\to X$ such that $T(f)(q)=\varepsilon(f)\a(q)f(\sigma(q))$ for every $f\in S(C_0(X,\mathbb{R}))$ and 
    $q\in Y$.
\end{abstract}
\maketitle

\section{Introduction}
The celebrated Wigner's  unitary-antiunitary theorem \cite{wig} can be stated as follows:
\textit{Let $T$ be a surjective, not necessarily linear, map  
between complex Hilbert spaces $H_1$ and $H_2$ which satisfies
\begin{equation}\label{1.1}
     \ab{\inn{T(f),T(g)}}=\ab{\inn{f,g}}
\end{equation}
for every $f,g\in H_1$.
Then there exist a unimodular function $\theta\colon H_1\to\C$ and a unitary or antiunitary operator $U$
such that $T=\theta U$.}
This result plays an important role in mathematical physics, and was proved 
in \cite{bargmann, geher, gyory, lomont}.

R\"atz \cite{ratz} generalized the Wigner theorem as follows:
Let $H_1$ and $H_2$ be real inner product spaces. 
If a map $T\colon H_1\to H_2$ satisfies \eqref{1.1},
then there exist a function $\theta\colon H_1\to \set{-1,1}$
and a linear isometry $U\colon H_1\to H_2$ such that
$T=\theta U$.
Following the R\"atz result,
Maksa and P\'ales \cite{mak} gave 
some equivalent conditions for \eqref{1.1}, 
provided that $H_1$ and $H_2$ are real inner product spaces. 
The following equality, which does 
not use inner products, 
is one of these conditions:
\begin{equation}\label{1.2}
    \set{\norm{T(f)+T(g)},\norm{T(f)-T(g)}}=\set{\norm{
    f+g},\norm{f-g}}
\end{equation}
for every $f,g\in H_1$.
Thus we can consider an analogue of the Wigner theorem in normed, not necessarily inner product, spaces. 
More concretely, one can define maps $T\colon H_1\to H_2$ satisfying \eqref{1.2}, provided that $H_1$ and $H_2$ are both real normed spaces: such maps are said to be \textit{phase-isometries}.

There are several results characterizing surjective 
phase-isometries between concrete normed spaces; 
see \cite{huangtan, jiatan, litan, mak, ZeHu}. 
Ili\v sevi\'c, Omladi\v c and Turn\v sek \cite{ili} proved a crucial result that every surjective phase-isometry $T$ between arbitrary real normed spaces $E_1$ and $E_2$ can be expressed as $T=\theta U$, where $\theta$ is a unimodular function on $E_1$ and $U$ is a surjective linear isometry between $E_1$ and $E_2$. 
Therefore, the study of surjective phase-isometries  on whole real normed spaces is essentially equivalent to the study of surjective isometries.
Recently, several papers concerning 
surjective phase-isometries between particular subsets
of real normed spaces have been published: see \cite{HMM, sun2, sun, tangao,TZH}.
In this paper, we focus on  surjective phase-isometries between the unit spheres of the spaces of continuous functions.

Throughout this paper, for a real Banach space $E$, the symbol $S(E)$ denotes 
the unit sphere of $E$. 
Let $L$ be a locally compact Hausdorff space. 
The symbol $\cl$ denotes the real Banach space of 
all continuous real-valued functions on $L$
vanishing at infinity, endowed with 
the supremum norm $\norm{h}=\sup_{x\in L}|h(x)|$.
The following is the main theorem of this paper.
\begin{thm}\label{MainTheorem}
    Assume that $X$ and $Y$ are locally compact Hausdorff spaces and that $T\colon\sx\to\sy$ is a surjective phase-isometry.
    Then there exist 
    a function $\varepsilon\colon\sx\to\set{-1,1}$, 
    a continuous function $\a\colon Y\to\set{-1,1}$ and
    a homeomorphism $\sigma\colon Y\to X$ such that
    \begin{equation*}
        T(f)(q)=\varepsilon(f)\a(q)f(\sigma(q))
    \end{equation*}
    for all $f\in \sx$ and $q\in Y$.
\end{thm}

\begin{rem}\label{rem1}
Tan and Gao gave a characterization of surjective phase-isometries between the unit spheres of $C(X, \R)$ and $C(Y, \R)$ on compact Hausdorff spaces $X$ and $Y$ (see \cite[Theorem 1]{tangao}). 
Our main result  does not assume the compactness of $X$ and $Y$. 
Thus it generalizes their result. 

On the other hand, Tan, Zhang and Huang investigated surjective phase-isometries from the unit sphere of a CL-space $E_1$ onto the unit sphere of an arbitrary Banach space $E_2$. 
They claimed that every surjective phase-isometry $T \colon S(E_1) \to S(E_2)$ can be expressed as $T = \theta U$, where $\theta$ is a unimodular function on $S(E_1)$ and $U \colon S(E_1) \to S(E_2)$ is a surjective isometry (see \cite[Theorem 3.4]{TZH}). 
Since $C_0(L, \R)$ is a CL-space as will be verified later, the result of Tan, Zhang and Huang, together with the Banach--Stone theorem, leads to  Theorem~\ref{MainTheorem}.
Thus our main result can be regarded as a corollary of their result.

However, it appears to the authors that there is a step in their proof that requires further justification. 
We furthermore provide a self-contained proof of Theorem~\ref{MainTheorem}, based on the ideas of Tan, Zhang and Huang \cite{TZH}. 
\end{rem}

\section{Preliminaries}
In the rest of this paper, we assume that all Banach spaces are real.
In this section, we present some lemmas on general properties of the phase-isometries and the maximal convex subsets of $S(\cl)$.

With the aid of the following result, we see that  
every surjective phase-isometry preserves the antipodal points.
Moreover, 
every surjective phase-isometry has its inverse. 
\begin{lem}[{Tan and Gao \cite[Lemma 2]{tangao}}]\label{lem2.0}
  Suppose that $E_1$ and $E_2$ are Banach spaces and $T\colon S(E_1)\to S(E_2)$ is a surjective phase-isometry.
  Then the identity $T(-f)=-T(f)$ holds for every $f\in S(E_1)$. 
  Moreover, $T$ is injective and 
  its inverse $T^{-1}$ is
  also a surjective phase-isometry. 
\end{lem}
Note that, by Lemma~\ref{lem2.0} and a straightforward calculation, $T(-A)=-T(A)$ holds
for every $A\subset S(E_1)$.

We introduce a concept which connects phase-isometries and isometries.

\begin{defn}
Let $E_1$ and $E_2$ be Banach spaces and let $U_1, U_2\colon S(E_1)\to S(E_2)$ be maps. 
We say that $U_1$ and $U_2$ are \textit{phase-equivalent}, and write $U_1\simeq U_2$, if there exists $\theta\colon S(E_1)\to\set{-1, 1}$ such that $U_1=\theta U_2$.  
\end{defn}

Note that the phase-equivalence $\simeq$ 
is an equivalence relation.
In the next lemma, using the relation $\simeq$, we provide a condition that characterizes when a map $U$ becomes a surjective phase-isometry.

\begin{lem}\label{piso}
    Suppose that $E_1$ and $E_2$ are Banach spaces.
    Let $T\colon S(E_1)\to S(E_2)$ be a surjective phase-isometry and let $\theta\colon S(E_1)\to\set{-1,1}$ be a function with $\theta(f)=\theta(-f)$ for every $f\in S(E_1)$.
    Then the map $U$ defined as $U=\theta T$ is also a surjective phase-isometry.
\end{lem}
\begin{proof}
Fix $f,g\in S(E_1)$. 
Having in mind that $\theta(f) \in \set{-1,1}$, we have  
\begin{align*}
 \norm{U(f) \pm U(g)} 
=  \norm{\theta(f) T(f) \pm \theta(g) T(g)}
= \norm{T(f) \pm \theta(f) \theta(g) T(g)}, 
\end{align*} 
where equalities hold with the same sign.
Since $\theta(f) \theta(g) \in \set{-1,1}$ and $T$ is a phase-isometry, the last equalities show that 
\begin{align*}
\set{\norm{U(f) + U(g)}, \norm{U(f) - U(g)}}
= &\set{\norm{T(f) + T(g)}, \norm{T(f) - T(g)}} \\
= &\set{\norm{f + g}, \norm{f - g}}.
\end{align*}
Hence the map $U$ is a phase-isometry.

To prove the surjectivity of $U$, fix $u \in S(E_2)$. 
The surjectivity of $T$ ensures the existence of $f_0 \in S(E_1)$ with $T(f_0) = u$. 
Set $c=\theta(f_0)\in\set{-1,1}$.
We see that 
\begin{align*}
\theta(c f_0) = c \quad \text{and} \quad 
T(c f_0) = c T(f_0). 
\end{align*}
Indeed, these equalities are trivial if $c = 1$, and follow from the assumption on $\theta$ and Lemma~\ref{lem2.0} if $c = - 1$. 
Hence we obtain 
\begin{align*}
U(c f_0) 
= \theta(c f_0)T(c f_0) 
= c^2 T(f_0) 
= T(f_0) 
= u, 
\end{align*}
which proves the surjectivity of $U$. 
\end{proof}

Every surjective phase-isometry between the unit spheres of Banach spaces preserves the union of each maximal convex set $M$ and its antipodal face $-M$.

\begin{lem}[{Tan, Zhang and Huang \cite[Proposition~2.5]{TZH}}]\label{lemTZH}
Let $E_1$ and $E_2$ be Banach spaces and let $T \colon S(E_1) \to S(E_2)$ be a surjective phase-isometry. 
For every maximal convex subset $M$ of $S(E_1)$, there exists a  maximal convex subset $N$ of $S(E_2)$ such that
\begin{align*}
T(M \cup -M) = N \cup -N. 
\end{align*}
\end{lem}

Motivated by the previous lemma, we shall characterize the maximal convex subsets of $S(\cl)$.
For each $(t,x)\in \set{-1,1}\times L$, we set 
\begin{equation*}
    tM_x=\set{f\in S(\cl):f(x)=t}.
\end{equation*}
To simplify notation, we write $M_x$ for $1M_x$ and $-M_x$ for $-1M_x$. 
Note that $M_x\cap-M_x=\emptyset$ for every $x\in L$.
Note also that $tM_x$ is a closed convex subset of $\sol$. 
In particular, it is closed and connected. 
In Lemma~\ref{HOM'}, we will prove that $tM_x$ is a maximal convex subset of $S(\cl)$.

The following lemma and its corollary will be used for characterizing the maximal convex subset of $S(\cl)$.
\begin{lem}\label{inclusion}
    If $(t, x), (s,y)\in\set{-1, 1}\times L$ satisfy $tM_{x}\subset sM_y$, then $(t, x)=(s, y)$. 
\end{lem}

\begin{proof}
    Let $(t, x), (s, y)\in \set{-1, 1}\times L$ satisfy $tM_x\subset sM_y$. 
    Suppose that $x\neq y$. 
    By Urysohn's lemma, there exists $f_0\in tM_x$
    such that $f_0(y)=0$, and thus $f_0\notin s M_y$.
    This contradicts the assumption that $tM_x\subset s M_y$. 
    Then $x=y$. 
    Taking $g_0\in tM_x \subset sM_y=s M_x$,  
    we obtain $t=g_0(x)=s$, and thus $(t, x)=(s, y)$. 
\end{proof}

\begin{cor}\label{inclusion2}
    If $x, y\in L$ satisfy $M_x\cup -M_x\subset M_y\cup -M_y$,  then $x=y$. 
\end{cor}
\begin{proof}
    Suppose that $x, y\in L$ satisfy $M_x\cup -M_x\subset M_y\cup -M_y$.
    Then $M_x\subset M_y\cup -M_y$. 
    Recall that $M_x$, $M_y$ and $-M_y$ are closed and connected.
    Since $M_y\cap -M_y=\emptyset$, it follows 
    that either 
    $M_x\subset M_y$ or $M_x\subset -M_y$ holds.
    Hence $x=y$ by Lemma~\ref{inclusion}.
\end{proof}

For any Banach space $E$, the symbol $B(E^*)$ denotes the closed unit ball of the dual space $E^*$ of $E$ and the symbol $\ext{B(E^*)}$ denotes the set of all extreme points of $B(E^*)$.
In general, the maximal convex subsets of $S(E)$ are deeply related to the extreme points of $B(E^*)$ as follows:

\begin{lem}[{Hatori, Oi and Shindo-Togashi \cite[Lemma 3.1]{HatoriOiShindo}}]\label{lemHOM}
    Let $E$ be a Banach space. 
    For every maximal convex subset $M$ of $S(E)$,
    there exists $\xi\in \ext{B(E^*)}$ such that $M=\xi^{-1}(\set{1})\cap S(E)$.
\end{lem}

The point evaluation at $x\in L$ is the map $\delta_x\colon \cl\to\R$ defined by $\delta_x(f)=f(x)$ for all $f\in\cl$. 
It is well-known that 
\begin{equation}\label{ext}
    \ext{B(\cl^*)}=\set{t\delta_x:(t,x)\in\set{-1, 1}\times L}, 
\end{equation}
(see \cite[Proposition~4.4.15]{DDLS}). 
With the aid of Lemma~\ref{lemHOM}, we see that every maximal convex subset of $S(\cl)$ is of the form $tM_x$ as follows:

\begin{lem}\label{Mchara}\label{HOM'}
	A subset $M$ of $S(\cl)$ is a maximal convex subset if and only if $M = t M_x$ for some $(t, x) \in \set{-1,1} \times L$. 
\end{lem}

\begin{proof}
Assume that $M$ is a maximal convex subset of $S(\cl)$. 
From Lemma~\ref{lemHOM}, there exists $\xi \in \ext{B(\cl^*)}$ such that $M = \xi^{-1}(\set{1}) \cap S(\cl)$. 
By \eqref{ext}, we can find some $(t, x)\in \set{-1,1} \times L$ such that $\xi = t \d_x$, and hence 
\begin{align*}
M = (t \d_x)^{-1}(\set{1}) \cap S(\cl) = \set{f \in S(\cl) : f(x) = t} = t M_x. 
\end{align*}

Conversely, fix $(t, x) \in \set{-1,1} \times L$. 
We show that $t M_x$ is a maximal convex subset of $S(\cl)$. 
Since $t M_x$ is a convex subset of $S(\cl)$, Zorn's lemma guarantees the existence of a maximal convex subset $M'$ of $S(\cl)$ such that $t M_x \subset M'$. 
By the previous paragraph, we can write $M' = s M_y$ for some $(s, y) \in \set{-1,1}\times L$. 
Since $tM_x\subset M' = sM_y$, it follows from Lemma~\ref{inclusion} that $(t, x)=(s, y)$, and hence $t M_x = s M_y = M'$. 
This proves that $t M_x$ is a maximal convex subset of $S(\cl)$. 
\end{proof}

Here we present two elementary propositions.
The former is routine, and the later follows from a standard application of Urysohn's lemma.
So we omit the proofs.
\begin{prop}\label{ConvNorm}
Let $M$ be a convex subset of the unit sphere $S(E)$ of a Banach space $E$ and let $f,g\in M$. 
Then 
$\norm{f + g} = 2$.
\end{prop}

% \begin{proof}
% Fix $f, g \in M$. 
% The convexity of $M$ shows that $(f + g) /2 \in M$. 
% Since $M$ is a subset of $S(E)$, we see that $\norm{f + g} = 2$, as desired.
% \end{proof}

\begin{prop}\label{Urysohn}
    Let $n\in\N$ and let $(t_i,x_i)\in\set{-1,1}\times L$ for $i=1,\dots, n$. 
    If $x_i\neq x_j$ for every $i\neq j$, then $\bigcap_{i=1}^nt_iM_{x_i}$ is non-empty. 
\end{prop}

\section{Rearrangement of $T$}
In the rest of this paper, let $X$ and $Y$ be locally compact Hausdorff spaces and assume that $T \colon \sx \to \sy$ is a surjective phase-isometry. 
By Lemma~\ref{HOM'}, every maximal convex subset of the unit sphere $\sx$ of $\cx$ can be expressed as $t M_x$ for some $(t, x) \in \set{-1,1} \times X$. 
In the same way, those of $\sy$ would also be denoted $s M_q$ for some $(s,q)\in \set{-1,1}\times Y$. 
However, to avoid confusion in notation, we use the symbol $s N_q$ instead of $s M_q$. 

By using the correspondence between the maximal convex subsets of $\sx$ and $\sy$, we obtain a suitable bijection between underlying spaces $X$ and $Y$.
\begin{lem}\label{lemtau}
There exists a bijection $\tau \colon X \to Y$ such that 
\begin{align}\label{eqtau}
T(M_x \cup -M_x) = N_{\tau(x)} \cup -N_{\tau(x)}
\end{align}
for every $x \in  X$. 
\end{lem}

\begin{proof}
Fix $x \in X$. 
We note that $M_x$ is a maximal convex subset of $\sx$ by Lemma~\ref{Mchara}. 
Applying Lemma~\ref{lemTZH} to $M_x$, we can find some maximal convex subset $N$ of $\sy$ such that 
$T(M_x \cup -M_x) = N \cup -N$.
By applying Lemma~\ref{Mchara} to $N$, there exists $\tau(x)\in Y$ such that $N \cup -N=N_{\tau(x)}\cup -N_{\tau(x)}$, and thus 
\begin{equation*}
    T(M_x \cup -M_x)=N_{\tau(x)} \cup -N_{\tau(x)}. 
\end{equation*}  
The point $\tau(x)$ is unique by Corollary~\ref{inclusion2}. 
Thus we obtain a map $\tau \colon X \to Y$ with \eqref{eqtau}.

We show the bijectivity of $\tau$. 
Note that the inverse map $T^{-1}$ of $T$ is also a surjective phase-isometry by Lemma~\ref{lem2.0}.
Then the same argument as in the previous paragraph guarantees the existence of a map $\sigma \colon Y \to X$ satisfying 
\begin{align*}
T^{-1}(N_q \cup -N_q) = M_{\sigma(q)} \cup -M_{\sigma(q)}
\end{align*}
for every $q \in Y$. 
Hence, for each $x \in X$, we deduce from \eqref{eqtau} that
\begin{align*}
M_x \cup -M_x 
& = T^{- 1} (T(M_x\cup -M_x)) 
= T^{- 1}(N_{\tau(x)}\cup -N_{\tau(x)}) \\
& = M_{\sigma(\tau(x))}\cup -M_{\sigma(\tau(x))}. 
\end{align*}
Corollary~\ref{inclusion2} gives $x = \sigma(\tau(x))$ for every $x\in X$. 
Interchanging the role of $T$ and $T^{- 1}$, we also see that $q = \tau(\sigma(q))$ for every $q\in Y$. 
Hence we conclude that $\tau$ is a bijection. 
\end{proof}

The map $T$ preserves the union $M_x \cup -M_x$ in the sense of \eqref{eqtau}. 
However, in general, it is not necessarily true that either $T(M_x) = N_{\tau(x)}$ or $T(M_x) = -N_{\tau(x)}$ holds. 
Nevertheless, by rearranging $T$, we will obtain a surjective phase-isometry $\Phi \colon \sx \to \sy$ with the following property: for every $x \in X$, either $\Phi(M_x) = N_{\tau(x)}$ or $\Phi(M_x) = -N_{\tau(x)}$. 
In what follows, we present some definitions and lemmas 
to define such a rearrangement of $T$. 

Let $x \in X$. 
We first construct a surjective phase-isometry $\Phi_x \colon \sx \to \sy$ which is phase-equivalent to $T$ and satisfies $\Phi_x(M_x) = N_{\tau(x)}$ and $\Phi_x(-M_x)=-N_{\tau(x)}$. 
Fix $f\in M_x\cup -M_x$.
Then there exists a unique $t\in \set{-1,1}$ such that $f\in tM_x$ since $M_x\cap -M_x=\emptyset$. 
Moreover, it follows from Lemma~\ref{lemtau} that $T(f)\in N_{\tau(x)}\cup -N_{\tau(x)}$. 
Hence there also exits a unique $s\in \set{-1,1}$ such that $T(f)\in sN_{\tau(x)}$. 
Thus, for every $f\in M_x\cup -M_x$, there corresponds a unique pair $(t,s)\in\set{-1,1}\times \set{-1,1}$ with $f\in tM_x$ and $ T(f)\in sN_{\tau(x)}$. 
Having in mind this correspondence, we define two maps $\theta_x$ and $\Phi_x$ as follows: 

\begin{defn}\label{def0}
Fix $x\in X$. 
For every $f\in M_x \cup -M_x$, there exists a unique pair $(t,s)\in\set{-1,1}\times \set{-1,1}$ with $f\in tM_x$ and $ T(f)\in sN_{\tau(x)}$. 
We set 
\begin{equation*}
    \theta_x(f)=ts\q (f\in M_x \cup -M_x).
\end{equation*} 
For every $f\notin M_x\cup -M_x$, we set $\theta_x(f)=1$.
We define a map $\theta_x\colon\sx\to\set{-1,1}$ by assigning $\theta_x(f)\in\set{-1,1}$ to each $f\in \sx$. 

We also define $\Phi_x\colon \sx\to \sy$ by 
\begin{equation*}
    \Phi_x(f)=\theta_x(f)T(f)
\end{equation*}
for every $f\in\sx$.
\end{defn}

We observe a property of the map $\theta_x$. 
Let $f\in M_x\cup -M_x$.
Then there exists a unique pair $(t,s)\in \set{-1,1}\times \set{-1,1}$ such that $f\in tM_x$ and $T(f)\in sN_{\tau(x)}$. 
Thus $\theta_x(f)=ts$ by the definition of $\theta_x$.
It follows from  Lemma~\ref{lem2.0} that $-f\in-tM_x$ and $T(-f)=-T(f)\in -sN_{\tau(x)}$. 
Then the definition of $\theta_x$ implies that $\theta_x(-f)=(-t)(-s)=\theta_x(f)$.
If $f\notin M_x\cup -M_x$, then we have $-f\notin M_x\cup -M_x$, which implies that $\theta_x(f)=1=\theta_x(-f)$.
Thus we obtain 
\begin{equation}\label{thetax}
    \theta_x(f)=\theta_x(-f)
\end{equation} 
for every $f\in \sx$.

The next lemma establishes that the map $\Phi_x$ indeed has desired properties.

\begin{lem}\label{pmax}
For each $x \in X$, the map $\Phi_x \colon \sx \to \sy$ is a surjective phase-isometry which is phase-equivalent to $T$ and satisfies $\Phi_x(M_x) = N_{\tau(x)}$ and $\Phi_x (-M_x) = -N_{\tau(x)}$. 
Moreover, for every $x,y\in X$, $\Phi_x$ and $\Phi_y$ are phase-equivalent and $\Phi_x(M_y \cup -M_y)=N_{\tau(y)} \cup -N_{\tau(y)}$ holds. 
\end{lem}

\begin{proof}
For every $x \in X$, it follows from the definition that  $\Phi_x = \theta_x T$, and thus it is obvious that $\Phi_x\simeq T$.
Hence, for every $x,y\in X$, we have $\Phi_x\simeq \Phi_y$ since $\simeq$ is an equivalence relation.  
Moreover, having in mind the definition of $\Phi_x$ and \eqref{thetax}, Lemma~\ref{piso} implies that $\Phi_x$ is a surjective phase-isometry. 

We prove that $\Phi_x(M_x)=N_{\tau(x)}$.
Let $f\in M_x$. 
Then we have $T(f)\in sN_{\tau(x)}$ for some $s\in \set{-1,1}$, which implies that $\theta_x(f)=s$.
The definition of $\Phi_x$ implies that 
\begin{equation*}
    \Phi_x(f)=\theta_x(f) T(f)=s T(f)\in N_{\tau(x)}.
\end{equation*}
Hence $\Phi_x(M_x)\subset N_{\tau(x)}$.
To prove the reverse inclusion $\Phi_x(M_x) \supset N_{\tau(x)}$, choose $u\in N_{\tau(x)}$ arbitrarily.
Then, by Lemma~\ref{lemtau}, there exists $f\in M_x\cup -M_x$ such that $T(f)=u$. 
Since $M_x\cap -M_x=\emptyset$, we have $f\in tM_x$ for some $t\in\set{-1,1}$, and thus $\theta_x(f)=t$.
Hence we obtain 
\begin{equation*}
    \Phi_x(f)=\theta_x(f)T(f)=tT(f)=tu,
\end{equation*} 
which implies that $\Phi_x(tf)=u$ by Lemma~\ref{lem2.0}.
It follows from $f\in tM_x$ that $tf\in M_x$, and thus $\Phi_x(M_x)\supset N_{\tau(x)}$.
Hence $\Phi_x(M_x)=N_{\tau(x)}$. 
By Lemma~\ref{lem2.0}, we obtain $\Phi_x(-M_x)=-\Phi_x(M_x)=-N_{\tau(x)}$.

% We prove that $\Phi_x(M_x) = N_{\tau(x)}$ and $\Phi_x (-M_x) = -N_{\tau(x)}$.  
% If $f \in M_x$, then we have $\Phi_x(f) = \theta_x(f) T(f) \in N_{\tau(x)}$ by the definition of $\theta_x$, and thus $\Phi_x(M_x) \subset N_{\tau(x)}$. 
% Lemma~\ref{lem2.0} shows that  $\Phi_x(-M_x)\subset - N_{\tau(x)}$. 
% To prove the reverse inclusion $\Phi_x(M_x) \supset N_{\tau(x)}$, choose $u\in N_{\tau(x)}$ arbitrarily. 
% Since $\Phi_x$ is a surjective phase-isometry, we deduce from Lemma~\ref{lemtau} that $\Phi_x(f) = u$ for some $f \in M_x \cup -M_x$. 
% If we had $f \in -M_x$, then we would have $u = \Phi_x(f) \in \Phi_x(-M_x) \subset -N_{\tau(x)}$, which contradicts $N_{\tau(x)} \cap -N_{\tau(x)} = \emptyset$. 
% Hence we obtain $\Phi_x(M_x) \supset N_{\tau(x)}$. 
% Lemma~\ref{lem2.0} shows that $\Phi_x(-M_x) \supset - N_{\tau(x)}$. 
% We thus conclude that $\Phi_x$ satisfies $\Phi_x(M_x) = N_{\tau(x)}$ and $\Phi_x (-M_x) = -N_{\tau(x)}$.

It remains to show that $\Phi_x(M_y\cup -M_y)=N_{\tau(y)}\cup -N_{\tau(y)}$.  
By the first paragraph, we see that $\Phi_x \simeq \Phi_y$, and thus there exists a function $\theta \colon \sx \to \set{-1, 1}$ such that $\Phi_x = \theta \Phi_y$. 
The previous paragraph yields $\Phi_y(M_y \cup -M_y) = N_{\tau(y)} \cup -N_{\tau(y)}$. 
Hence, for every $f \in M_y\cup -M_y$, we have  
\begin{align*}
\Phi_x(f) = \theta(f) \Phi_y(f)\in 
\theta(f) (N_{\tau(y)} \cup -N_{\tau(y)}) 
= N_{\tau(y)} \cup -N_{\tau(y)}, 
\end{align*}
which proves $\Phi_x(M_y \cup -M_y) \subset N_{\tau(y)} \cup -N_{\tau(y)}$.  
Conversely, choose $u\in N_{\tau(y)} \cup -N_{\tau(y)}$. 
By the previous paragraph, there exists $f_0\in M_y\cup -M_y$ such that $\Phi_y(f_0)=u$. 
Set $c=\theta(f_0)$.
By the choice of the map $\theta$ and Lemma~\ref{lem2.0}, we have $\Phi_x(cf_0)=c\Phi_x(f_0)=\Phi_y(f_0)=u$. 
It follows from $c\in\set{-1,1}$ that $cf_0\in M_y\cup -M_y$.
Thus we conclude that$\Phi_x(M_y \cup -M_y) \supset N_{\tau(y)} \cup -N_{\tau(y)}$, which implies that $\Phi_x(M_y\cup -M_y)=N_{\tau(y)}\cup -N_{\tau(y)}$, as claimed. 
\end{proof}

By gluing the surjective phase-isometries $\Phi_x$, we will define a surjective phase-isometry $\Phi \colon \sx \to \sy$ which is phase-equivalent to $T$ and satisfies that, for each $x\in X$, either $\Phi(M_x) = N_{\tau(x)}$ or $\Phi(M_x) = -N_{\tau(x)}$. 
More explicitly, we will construct a suitable function $\rho \colon X \to \set{-1,1}$ and define $\Phi \colon \sx \to \sy$ by $\Phi(f) = \rho(x) \Phi_x(f)$ on  each $M_x \cup -M_x$. 
Since this definition depends on the choice of the maximal convex subset to which $f$ belongs, we need to show that $\Phi$ is well-defined. 
Hence we shall prove several lemmas to ensure 
the well-definedness of the rearrangement $\Phi$ of $T$.

\begin{lem}\label{lem3.17}
Let $(t, x), (s, y) \in \set{-1, 1}\times X$ with $t M_x \cap s M_y \ne \emptyset$. 
Then either $\Phi_x(t M_x \cap s M_y) \subset N_{\tau(y)}$ or $\Phi_x(t M_x \cap s M_y) \subset -N_{\tau(y)}$ holds. 
\end{lem}

\begin{proof}
Fix $(t, x), (s, y) \in \set{-1, 1}\times X$ with $t M_x \cap s M_y \ne \emptyset$. 
Since $tM_x\cap sM_y\subset M_y\cup -M_y$, the second part of Lemma~\ref{pmax} implies that 
\begin{equation*}
    \Phi_x(tM_x \cap sM_y)\subset N_{\tau(y)} \cup -N_{\tau(y)}.
\end{equation*}
We shall show that $\Phi_x$ is continuous on $tM_x$. 
Fix $f,g\in tM_x$.
Then $\norm{f+g}=2$ by Proposition~\ref{ConvNorm}. 
It follows from Lemma~\ref{pmax} that $\Phi_x(f),\Phi_x(g)\in tN_{\tau(x)}$.
Thus $\norm{\Phi_x(f)+\Phi_x(g)}=2$ by Proposition~\ref{ConvNorm}. 
Since $\Phi_x$ is a phase-isometry, we must have $\norm{\Phi_x(f)-\Phi_x(g)}=\norm{f-g}$.  
Hence $\Phi_x$ is continuous on $tM_x$. 
Since $tM_x\cap sM_y$ is convex, it must be connected.
Thus the continuity of $\Phi_x$ implies that the set $\Phi_x(tM_x\cap sM_y)$ is also connected.   
Having in mind that $N_{\tau(y)}$ and $-N_{\tau(y)}$ are closed and disjoint, we must have either $\Phi_x(tM_x\cap sM_y)\subset N_{\tau(y)}$ or $\Phi_x(tM_x\cap sM_y)\subset -N_{\tau(y)}$. 
\end{proof}

\begin{lem}\label{lem3.18}
For each $x,y \in X$ and each $s\in\set{-1,1}$, there exists a unique $r\in \set{-1, 1}$ such that 
\begin{align}\label{uniquer}
\Phi_x(M_x \cap sM_y) \cup \Phi_x(-M_x \cap sM_y) \subset r N_{\tau(y)}.
\end{align}
\end{lem}
\begin{proof}
First, we consider the case when $x = y$. 
Then exactly one of the sets $M_x \cap sM_y$ and $-M_x \cap sM_y$ is empty, and the other equals $sM_y$.
Hence $\Phi_x(M_x \cap sM_y) \cup \Phi_x(-M_x \cap sM_y)=\Phi_y(sM_y)$.
It follows from Lemma~\ref{pmax} that $\Phi_y(sM_y)\subset sN_{\tau(y)}$.
Thus we obtain \eqref{uniquer} setting $r=s$.

We next assume that $x \ne y$. 
Note that, by Proposition \ref{Urysohn}, $M_x \cap s M_y \ne \emptyset$ and $- M_x \cap s M_y \ne \emptyset$. 
By Lemma~\ref{lem3.17}, we can find $r,r' \in \set{-1, 1}$ such that 
\begin{equation*}
   \Phi_x(M_x \cap s M_y) \subset r N_{\tau(y)}\qbox{and}\Phi_x(-M_x \cap s M_y) \subset r' N_{\tau(y)}. 
\end{equation*}
The bijectivity of $\tau$ implies that $\tau(x)\neq \tau(y)$.
Hence it follows from Proposition~\ref{Urysohn} that $-N_{\tau(x)} \cap r N_{\tau(y)} \neq \emptyset$.  
Choose 
\begin{equation*}
    u_0 \in - N_{\tau(x)} \cap rN_{\tau(y)}.
\end{equation*}
To prove $r=r'$, we show that $u_0\in r'N_{\tau(y)}$.
The surjectivity of $\Phi_x$ guarantees the existence of  $f_0\in\sx$ such that $\Phi_x(f_0)=u_0$. 
Since $\Phi_x(f_0)=u_0\in -N_{\tau(x)}$, Lemma~\ref{pmax} implies that   
\begin{equation*}
    f_0\in -M_x.
\end{equation*}
% By the choice of $u_0$, we have $\Phi_x(f_0)=u_0\in -N_{\tau(x)}$. 
% Then 
% \begin{equation*}
%     f_0\in -M_x
% \end{equation*}by Lemma~\ref{pmax}.
On the other hand, since $T\simeq \Phi_x$, we have either $T(f_0)=\Phi_x(f_0)$ or $T(f_0)=-\Phi_x(f_0)$.
Then it follows from $\Phi_x(f_0)\in rN_{\tau(y)}$ that $T(f_0)\in N_{\tau(y)} \cup -N_{\tau(y)}$,
and hence $f_0 \in M_y \cup - M_y$ by Lemma~\ref{lemtau}.
In particular, we have $f_0\in sM_y\cup -sM_y$.
Thus either $f_0 \in - M_x\cap s M_y$ or $f_0 \in - M_x\cap -s M_y$ holds.
If $f_0$ were in $ - M_x\cap -s M_y$, then $u_0=\Phi_x(f_0)$ would be in $-rN_{\tau(y)}$ by the choice of $r$ and Lemma~\ref{lem2.0}.
This contradicts $u_0\in r N_{\tau(y)}$.
Thus we obtain $f_0 \in - M_x\cap s M_y$. 
The choice of $u_0$ and $r'$ shows that $u_0 = \Phi_x(f_0) \in rN_{\tau(y)}\cap r'N_{\tau(y)}$, which implies that $r=r'$.
Hence we conclude that $\Phi_x(M_x\cap sM_y)\cup \Phi_x(-M_x \cap s M_y)\subset r N_{\tau(y)}$.
\end{proof}

Fix an arbitrary point $p \in X$. 
We use the point $p$ as the base point to define a function $\rho \colon X \to \set{-1, 1}$, which was mentioned before.

\begin{defn}\label{defrho}
We define a map $\rho \colon X \to \set{-1, 1}$ by assigning to each $x \in X$ the unique value $r \in \set{-1, 1}$ satisfying \eqref{uniquer}: 
\begin{align}\label{RHO}
\Phi_x(M_x \cap M_p) \cup \Phi_x(-M_x \cap M_p) \subset \rho(x) N_{\tau(p)}. 
\end{align}
\end{defn}
Note that $\rho(p)=1$ by Lemma~\ref{pmax}.

We will define the rearrangement $\Phi$ of $T$ by $\Phi = \rho(x) \Phi_x$ on each $M_x \cup -M_x$, as mentioned before. 
The difficulty is that $\Phi(f)$ may be defined in two different ways when $f$ belongs to the intersection of two maximal convex subsets $t M_x$ and $s M_y$. 
Hence we need to ensure that the rearrangement $\Phi$ of $T$ is well-defined.

\begin{lem}\label{lem3. 19}
Let $(t, x), (s, y)\in\set{-1, 1}\times X$ with $t M_x \cap s M_y \neq \emptyset$. 
Then 
\begin{align*}
\rho(x) \Phi_x = \rho(y) \Phi_y \q \text{on}\hspace{0.5em} tM_x\cap s M_y.
\end{align*} 
\end{lem}

\begin{proof}
Fix $(t,x), (s,y)\in\set{-1,1}\times X$ with $tM_x\cap s M_y\neq \emptyset$.
We first claim that there exists $r\in\set{-1,1}$ such that $\Phi_x=r\Phi_y$ on $tM_x\cap s M_y$.
By Lemma~\ref{lem3.18}, we can find $r_0\in\set{-1,1}$ such that 
\begin{equation}\label{rzero}
    \Phi_x(M_x\cap s M_y)\cup \Phi_x(-M_x\cap s M_y)\subset r_0 N_{\tau(y)}.
\end{equation}
Since $\Phi_x\simeq\Phi_y$, there exists a function $\theta\colon\sx\to\set{-1,1}$ such that $\Phi_x=\theta\Phi_y$.
Note that $r_0 s\in\set{-1,1}$.
We show that $\theta=r_0 s$ on $tM_x\cap s M_y$.
% Suppose, on the contrary, that there exists $f_0\in tM_x\cap s M_y$ such that $\theta(f_0)=-r_0 s$.
% Then $\Phi_x(f_0)=-r_0 s\Phi_y(f_0)$.
% The choice of $r_0$ shows that $\Phi_x(f_0)\in r_0 N_{\tau(y)}$. 
% On the other hand, since $f_0\in s M_y$, we have $-r_0 sf_0\in -r_0M_y$. 
% Then it follows from Lemma~\ref{pmax} that 
% $\Phi_x(f_0)=-r_0 s\Phi_y(f_0)=\Phi_y(-r_0 sf_0)\in-r_0 N_{\tau(y)}$, which contradicts $\Phi_x(f_0)\in r_0 N_{\tau(y)}$.
Fix $f \in t M_x \cap s M_y$ arbitrarily. 
Then we have $\Phi_x(f) \in r_0 N_{\tau(y)}$ by \eqref{rzero}. 
On the other hand, since $f \in s M_y$, it follows from Lemma 3.3 that $\Phi_y(f) \in \Phi_y(s M_y) = s N_{\tau(y)}$. 
Hence we obtain 
$$
r_0 = \Phi_x(f)(\tau(y)) 
= \theta(f) \Phi_y(f)(\tau(y)) 
= \theta(f) s. 
$$
This proves that $\theta(f) = r_0 s$ for every $f \in t M_x \cap s M_y$.
By the choice of $\theta$, we conclude that 
\begin{equation}\label{3.7.1}
    \Phi_x=r\Phi_y \q \text{on}\hspace{0.5em} tM_x\cap s M_y,
\end{equation}
where $r=r_0 s\in\set{-1,1}$.

Since the lemma is trivial if $x = y$, we assume that $x \ne y$. 
Then the set $\set{x, y, p}$ consists of two or three points. 
Assume first that $\set{x, y, p}$ consists of two points. 
Since $x \ne y$, either $x = p$ or $y = p$ occurs. 
We first consider the case when $y=p$.
Suppose that $s=1$. 
Substituting $y=p$ and $s=1$ into \eqref{rzero} and comparing this with \eqref{RHO}, we see that  $r_0=\rho(x)$, and hence $r=r_0 s=\rho(x)$. 
Thus it follows from \eqref{3.7.1} and $\rho(p)=1$ that
\begin{equation}\label{3.7.2}
    \rho(x)\Phi_x=\rho(p)\Phi_p\q \text{on} \hspace{0.5em} tM_x\cap M_p.
\end{equation}
Suppose that $s=-1$ and fix $f \in t M_x \cap -M_p$. 
Then $-f \in - t M_x \cap M_p$. 
It follows from Lemma~\ref{lem2.0} and \eqref{3.7.2} that  
\begin{align*}
- \rho(x) \Phi_x(f)
= \rho(x) \Phi_x(-f) 
= \rho(p) \Phi_p(-f) 
= - \rho(p) \Phi_p(f). 
\end{align*}
Then $\rho(x) \Phi_x(f) = \rho(p) \Phi_p(f)$, which implies that $\rho(x) \Phi_x = \rho(p) \Phi_p$ on $t M_x \cap - M_p$. 
In the same way, we can show the lemma in the case when $x=p$.
Thus the lemma has been proved in the case that the set $\set{x,y,p}$ consists of two points.

Assume next that $\set{x, y, p}$ consists of three points. 
Proposition~\ref{Urysohn} guarantees the existence of $f_0\in tM_x\cap s M_y\cap M_p$.
It follows from  \eqref{3.7.2}  and \eqref{3.7.1} that 
\begin{align*}
    \rho(x)\Phi_x(f_0)=\Phi_p(f_0)=\rho(y)\Phi_y(f_0)\qbox{and}
    \Phi_x(f_0)=r\Phi_y(f_0).
\end{align*}
Thus we obtain $\rho(x)r\Phi_y(f_0)=\rho(y)\Phi_y (f_0)$, which implies that $r=\rho(x)\rho(y)$.
Therefore, we deduce from \eqref{3.7.1} that $\rho(x)\Phi_x=\rho(y)\Phi_y$ on $tM_x\cap sM_y$, as desired.
\end{proof}

Now we are in a position to define the rearrangement $\Phi$ of $T$.	 
Lemma~\ref{lem3. 19} ensures the well-definedness of the following definition. 
Note that, for every $f\in \sx$, there exists at least one $x\in X$ such that $|f(x)|=\norm{f}$. In particular, for every $f\in \sx$, there exists $x\in X$ such that $f\in M_x\cup -M_x$.

\begin{defn}\label{defn3. 20}
Define a map $\Phi \colon \sx \to \sy$ as follows:
for $f \in \sx$, if $f \in M_x \cup -M_x$ for some $x \in X$, then we set
\[
\Phi(f) = \rho(x)\Phi_x(f).
\]
In other words, 
\begin{align*}
\Phi = \rho(x) \Phi_x \quad \text{on} \hspace{0.5em} M_x \cup -M_x 
\end{align*}
for every $x \in X$. 
\end{defn}

\begin{lem}\label{Pmax}
The map $\Phi$ is a surjective phase-isometry and is phase-equivalent to $T$. 
Moreover,  
\begin{align*}
\Phi(M_x) = \rho(x) N_{\tau(x)} \quad \text{and} \quad 
\Phi(- M_x) = - \rho(x) N_{\tau(x)} 
\end{align*}
for every $x \in X$. 
\end{lem}

\begin{proof}
% Recall that $\Phi_x = \theta_x T$ by Definition~\ref{def0},
% and hence $\Phi=\rho(x) \theta_x T$ on $M_x\cup -M_x$. 
% We first claim that, for each $f \in \sx$, the value $\rho(x) \theta_x(f)$ is independent of the choice of $x \in X$ such that $f \in M_x\cup -M_x$.  
% Fix $f \in \sx$, and assume that $f \in t M_x \cap s M_y$ for some $(t, x), (s, y) \in \set{-1, 1} \times X$. 
% The definition of $\Phi$ implies that
% \begin{align*}
% \Phi(f) 
% = \rho(x) \Phi_x(f) 
% = \rho(x) \theta_x(f) T(f). 
% \end{align*}
% Similarly, we have $\Phi(f) = \rho(y) \theta_y(f) T(f)$. 
% Hence  $\rho(x) \theta_x(f) T(f) = \rho(y) \theta_y(f) T(f)$, which implies that $\rho(x) \theta_x(f) = \rho(y) \theta_y(f)$. 

We first claim that, for each $f \in \sx$, the value $\rho(x) \theta_x(f)$ is independent of the choice of $x \in X$ such that $f \in M_x\cup -M_x$.  
Fix $f \in \sx$, and assume that $f \in t M_x \cap s M_y$ for some $(t, x), (s, y) \in \set{-1, 1} \times X$. 
The definitions of $\Phi$ and $\Phi_x$ imply that
\begin{align*}
\Phi(f) 
= \rho(x) \Phi_x(f) 
= \rho(x) \theta_x(f) T(f). 
\end{align*}
Similarly, we have $\Phi(f) = \rho(y) \theta_y(f) T(f)$. 
Hence  $\rho(x) \theta_x(f) T(f) = \rho(y) \theta_y(f) T(f)$, which implies that $\rho(x) \theta_x(f) = \rho(y) \theta_y(f)$.

Let us show that $\Phi$ is a surjective phase-isometry.
Define a map $\theta \colon \sx \to \set{-1, 1}$ by 
\begin{align*}
\theta(f) = \rho(x) \theta_x(f), 
\end{align*}
where $x \in X$ is a point such that $f \in M_x\cup -M_x$. 
The previous paragraph guarantees that the map $\theta$ is well-defined. 
For each $f \in \sx$, we can find $x \in X$ such that $f\in  M_x\cup -M_x$, and thus 
\begin{align*}
\Phi(f) = \rho(x) \Phi_x(f) = \rho(x) \theta_x(f) T(f) = \theta(f) T(f). 
\end{align*} 
Hence we obtain $\Phi = \theta T$.
% , which shows that $\Phi \simeq T$.
Moreover, the definition of $\theta$ and \eqref{thetax} yield $\theta(f) = \theta(- f)$. 
Combining the last two equalities with Lemma~\ref{piso}, we see that $\Phi$ is a surjective phase-isometry and $\Phi\simeq T$.

It remains to show that $\Phi(M_x)=\rho(x)N_{\tau(x)}$ and $\Phi(-M_x)=-\rho(x)N_{\tau(x)}$.
Let $x \in X$ and $f \in M_x$. 
It follows from Lemma~\ref{pmax} that $\Phi_x (M_x)=N_{\tau(x)}$. 
The definition of $\Phi$ implies that $\Phi = \rho(x) \Phi_x$ on $M_x$, and consequently 
\begin{align*}
\Phi(M_x) 
= \rho(x) \Phi_x(M_x) 
= \rho(x) N_{\tau(x)}. 
\end{align*}
Since $\Phi$ is a surjective phase-isometry, it preserves the antipodal points by Lemma~\ref{lem2.0}.
Hence we obtain $\Phi(- M_x) = - \Phi(M_x) = - \rho(x) N_{\tau(x)}$. 
This completes the proof. 
\end{proof}

\section{The proof of the main theorem}
We first establish a proposition to prove the main theorem.
Using this proposition, we determine the form of the surjective phase-isometry $T$.
The proof of the following proposition is based on \cite[Lemma~2.17]{cue}.  

\begin{prop}\label{addbishop}
    Suppose that $f\in\sx$ and $x_0\in X$. 
    Suppose further that $x'\in  X$ with $|f(x')|=\norm{f}$. 
    Define a constant $t\in\set{-1, 1}$ by 
    \begin{equation*}
        t=\begin{dcases}
            \frac{f(x_0)}{|f(x_0)|}&\mathrm{if}\ f(x_0)\neq0,\\
            1&\mathrm{if}\ f(x_0)=0.
        \end{dcases}
    \end{equation*}
     Under the above assumptions, if $|f(x_0)|<1$,
     then there exists $g\in M_{x_0}$ such that
    \begin{equation*}
        f+(1 -|f(x_0)|)tg\in t M_{x_0}\cap f(x')M_{x'}. 
    \end{equation*}
\end{prop}

\begin{proof}
    Put $\d=(1-|f(x_0)|)/2\in(0, 1)$. 
    We define closed subsets $F_0$ and $F_n$ of $X$ by
    \begin{equation*}
        F_0=\set{x\in X:\frac{\d|f(x_0)-f(x')|}{2}\leq |f(x_0)-f(x)|}, 
    \end{equation*}
    \begin{equation*}
        F_n=\set{x\in X:\frac{\d|f(x_0)-f(x')|}{2^{n+1}}\leq|f(x_0)-f(x)|\leq \frac{\d|f(x_0)-f(x')|}{2^n}} 
    \end{equation*}
    for each $n\in \N$.
    By definition, we have $x_0\notin F_0\cup F_n$ for all $n\in \N$.
    Thus, for every $n\in \N$, there exists $g_n\in M_{x_0}$ such that $g_n(F_0\cup F_n)=\set{0}$ by Urysohn's lemma. 
    Set $g=\sum_{n=1}^{\infty}g_n/2^n$.
    Since $g(x_0)=1$, we have $g\in M_{x_0}$.
    Set $t\in\set{-1,1}$ as in the statement and define 
    \begin{equation*}
        h=f+(1-|f(x_0)|)tg.
    \end{equation*}
    We first show that $\norm{h}=1$.  
    Fix $x\in X$. 
    We need to consider the following three cases:
    $x\in F_0$, $x\in F_{k}$ for some $k\in\N$ and 
    $x\in X\setminus \bigcup_{m\in\N\cup\set{0}}F_m$.
 
    If $x\in F_0$, then $g_n(x)=0$ for all $n\in \N$, and thus $g(x)=0$. 
    This implies that $|h(x)|=|f(x)|\leq1$.  
    
    Suppose that $x\in F_{k}$
    for some $k\in\N$. 
    Since $g_{k}(x)=0$ and $|g_n(x)|\leq 1$ for all $n\in \N\setminus \set{k}$, it follows that 
    \begin{equation*}
        |g(x)|\leq \sum_{n\neq k}\frac{|g_n(x)|}{2^n}\leq \sum_{n\neq k}\frac{1}{2^n}=1-\frac{1}{2^k}.
    \end{equation*}
    We deduce from the definition of $F_k$ that $|f(x)|\leq \d|f(x_0)-f(x')|/2^k+|f(x_0)|$.
    Having in mind that the definition of $\delta$ and $|t_0|=1$, these two inequalities show that
    \begin{align*}
        |h(x)|&\leq |f(x)|+(1-|f(x_0)|)|g(x)|\\ 
        &\leq \bigg(\frac{\d |f(x_0)-f(x')|}{2^k}+|f(x_0)|\bigg)+(1-|f(x_0)|)\bigg(1-\frac{1}{2^k}\bigg)\\
        &\leq 1+\frac{\d}{2^k}(|f(x_0)-f(x')|-2). 
    \end{align*}
    % The most right-hand side of the last inequality is less than $1$ since $|f(x_0)-f(x')|<2$ by $|f(x_0)|<1$. 
    The right-hand side of the last inequality is less than $1$, since $|f(x_0)-f(x')|<2$, which follows from the assumption $|f(x_0)|<1$.
    Thus $|h(x)|\leq 1$. 
        
    If $x\notin F_m$ for all $m\in \N\cup\set{0}$, then it follows from the definition of $F_m$ that $f(x)=f(x_0)$. 
    Thus we obtain 
    \begin{equation*}
        |h(x)|\leq |f(x_0)|+1-|f(x_0)|=1,
    \end{equation*}
    which implies that $|h(x)|\leq1$.  
    
    We have shown that $|h(x)|\leq 1$ for every $x\in X$, and hence $\norm{h}\leq1$. 
    Moreover, by $g\in M_{x_0}$ and the definition of $t$, we have   
    $h(x_0)=f(x_0)+(1-|f(x_0)|)tg(x_0)=t$.
    Thus $\norm{h}=1$, which implies that   
    $h\in tM_{x_0}$. 
    We show that $h\in f(x')M_{x'}$. 
    It follows from $\d\in(0,1)$ that
    $\delta|f(x_0)-f(x')|/2
    <|f(x_0)-f(x')|$, and then $x'\in F_0$.
    This gives $g_n(x')=0$ for all $n\in \N$,
    which implies that $g(x')=0$. 
    The last equality 
    shows that $h(x')=f(x')$, and hence $h\in f(x')M_{x'}$.
    We thus conclude that 
    $h\in t M_{x_0}\cap f(x')M_{x'}$.  
\end{proof}

Now we are ready to prove the main theorem.
We determine the form of the surjective phase-isometry $T$ after determining that of $\Phi$, the rearrangement of $T$.

\begin{proof}[\textbf{Proof of Theorem~\ref{MainTheorem}}]
    We divide the proof into four steps and first determine the form of $\Phi$. 
We define $\sigma \colon Y \to X$ to be the inverse map $\tau^{-1}$ of $\tau$ and define $\a \colon Y \to \set{-1,1}$ by $\a = \rho \circ \sigma$. 

\begin{step}\label{stepA}
Let $f \in \sx$ and let $(t, q) \in \set{-1, 1} \times Y$. 
Then 
\begin{align}\label{EqK}
f \in t M_{\sigma(q)} \quad \text{if and only if} \quad \Phi(f) \in \a(q) t N_q. 
\end{align}
In particular, $|f(\sigma(q))| = 1$ if and only if $|\Phi(f)(q)| = 1$. 
\end{step}

\begin{sproof}
Fix $f \in \sx$ and $(t, q) \in \set{-1, 1} \times Y$. 
Lemma~\ref{Pmax} shows that 
\begin{align*}
\Phi(t M_{\sigma(q)}) 
= \rho(\sigma(q)) t N_{\tau(\sigma(q))}
= \a(q) t N_q, 
\end{align*}
which proves \eqref{EqK}. 
\end{sproof}

To determine the form of $\Phi$, fix $f \in \sx$ and $q \in Y$. 
If $|f(\sigma(q))| = 1$, then $f \in t M_{\sigma(q)}$ for some $t \in \set{-1, 1}$, and hence $\Phi(f) \in \a(q) t N_q$. 
In this case, we see that 
\begin{align*}
\Phi(f)(q) = \a(q) t = \a(q) f(\sigma(q)). 
\end{align*}
The next purpose is to show that the last equality holds even if $|f(\sigma(q))| < 1$. 
Hence we assume that $|f(\sigma(q))| < 1$.

Set $t_0\in\set{-1,1}$ by 
    \begin{equation*}
        t_0=\begin{cases}
            \frac{f(\sigma(q))}{|f(\sigma(q))|}& \text{if}\ f(\sigma(q))\neq 0,\\
            1&\text{if}\ f(\sigma(q))=0.
        \end{cases}
    \end{equation*}
    Choose $x'\in X$ so that $f\in f(x')M_{x'}$. 
    Applying Proposition~\ref{addbishop} with $\sigma(q)$ instead of $x_0$,  
    there exists $g\in M_{\sigma(q)}$ such that 
    \begin{equation*}
        h=f+(1 -|f(\sigma(q))|)t_0g\in t_0 M_{\sigma(q)}\cap f(x')M_{x'}.
    \end{equation*}

\begin{step}\label{stepB}
The equality 
    \begin{equation}\label{4.7}
        |\Phi(f)(q)|=|f(\sigma(q))|
    \end{equation}
holds. 
\end{step}

\begin{sproof}
Since $h,f\in f(x')M_{x'}$, it follows from Lemma~\ref{Pmax} that $\Phi(h),\Phi(f)\in \rho(x')f(x')N_{\tau(x')}$. 
Hence Proposition~\ref{ConvNorm} implies that $\norm{\Phi(h)+\Phi(f)}=2 = \norm{h+f}$. 
    Since $\Phi$ is a phase-isometry, we must have  
    \begin{equation}\label{isometry}
        \norm{\Phi(h)-\Phi(f)}=\norm{h-f}.
    \end{equation}
Let us show that
    $|f(\sigma(q))|\leq|\Phi(f)(q)|$. 
    Since $h\in t_0M_{\sigma(q)}$, \eqref{EqK} yields $\Phi(h)\in\a(q)t_0N_q$. 
By the triangle inequality, we have 
\begin{align*}
1-|\Phi(f)(q)| 
= |\Phi(h)(q)|-|\Phi(f)(q)|
\leq |\Phi(h)(q)-\Phi(f)(q)|
\leq\norm{\Phi(h)-\Phi(f)}.
\end{align*}
It follows from  \eqref{isometry} and $\norm{t_0g}=1$ that   
\begin{align*}
\norm{\Phi(h)-\Phi(f)}
= \norm{h-f}
= \norm{(1-|f(\sigma(y))|)t_0g}
= 1-|f(\sigma(q))|. 
\end{align*}
Combining the last inequality and the last equality, we have 
\begin{align}\label{IneqK}
1-|\Phi(f)(q)| 
= |\Phi(h)(q)|-|\Phi(f)(q)|
\leq |\Phi(h)(q)-\Phi(f)(q)|
=1-|f(\sigma(q))|, 
\end{align}
which implies that $|f(\sigma(q))|\leq|\Phi(f)(q)|$. 

Note that $|\Phi(f)(q)|<1$ by Step~\ref{stepA}. 
Applying the same argument to  $(\Phi^{-1},\Phi(f),\sigma(q))$ instead of $(\Phi,f,q)$, we also obtain $|\Phi(f)(\tau(\sigma(q)))|\leq |\Phi^{-1}(\Phi(f))(\sigma(q))|$.
    Having in mind that $\sigma=\tau^{-1}$, the last inequality can be rewritten as $|\Phi(f)(q)|\leq |f(\sigma(q))|$.
    We thus obtain \eqref{4.7}.
\end{sproof}

\begin{step}\label{stepC}
The equality 
\begin{align}\label{EqL}
\Phi(f)(q) = \a(q) f(\sigma(q)). 
\end{align}
holds. 
\end{step}

\begin{sproof}
If $\Phi(f)(q)=0$, then $f(\sigma(q))=0$ by \eqref{4.7}.
Hence \eqref{EqL} holds.
Suppose next that $\Phi(f)(q)\neq 0$. 
% Let $h$ be the function defined before Step~\ref{stepB}.
% By the choice of $h$ and \eqref{4.7}, we obtain $|\Phi(h)(y)| = |h(\sigma(y))| = 1$.
% Since we assumed $|f(\sigma(y))|<1$, \eqref{4.7} implies that $|\Phi(f)(y)|<1$.
% These two facts yield $\Phi(h)(y)-\Phi(f)(y)\neq 0$.
It follows from \eqref{4.7} and \eqref{IneqK} that 
\begin{equation*}
    |\Phi(h)(q)|-|\Phi(f)(q)|=|\Phi(h)(q)-\Phi(f)(q)|.
\end{equation*}
Having in mind that $\Phi(f)(q)\neq 0$, 
by the equality condition of the triangle inequality, there exists $c> 0$ such that $\Phi(f)(q)=c\Phi(h)(q)$.
Since $h \in t_0 M_{\sigma(q)}$, we have $\Phi(h) \in \a(q) t_0 N_q$ by \eqref{EqK}, and hence 
\begin{align}\label{EqPreComp}
\Phi(f)(q) = c  \a(q) t_0. 
\end{align}
By taking absolute values of the above equality, \eqref{4.7} implies that $c = |\Phi(f)(q)| = |f(\sigma(q))|$. 
Substituting this into \eqref{EqPreComp}, we see that 
\begin{align*}
\Phi(f)(q)=|f(\sigma(q))| \a(q) t_0 = \a(q) f(\sigma(q)), 
\end{align*}
where the last equality follows from the definition of the constant $t_0$.  
\end{sproof}

We have proven that equality \eqref{EqL} holds for every $f \in \sx$ and $q \in Y$. 

\begin{step}\label{stepD}
The map $\sigma$ is a homeomorphism and the function $\a$ is continuous. 
\end{step}

\begin{sproof}
 Fix an open subset $U$ of $X$ and 
    $q_0\in\sigma^{-1}(U)$. 
    By Urysohn's lemma, there exists $f_0\in\sx$
    such that $f_0(\sigma(q_0))=1$ and
     $f_0(X\setminus U)=\set{0}$.  
    We set an open subset $V$ of $Y$ by
    \begin{equation*}
    V=\set{q\in Y:|\Phi(f_0)(q)|>\frac{1}{2}}.
    \end{equation*}
    The equality \eqref{EqL} gives $|\Phi(f_0)(q_0)|=|f_0(\sigma(q_0))|=1>1/2$, 
    and then $q_0\in V$.  
    Fix $q\in V$. 
    It follows from \eqref{EqL} and the definition of $V$ that 
    $|f_0(\sigma(q))|=|\Phi(f_0)(q)|>1/2$,
    and thus $f_0(\sigma(q))\neq0$.  
    This shows that $\sigma(q)\in U$.  
    Since $q\in V$ is arbitrarily chosen,
    we obtain $V\subset \sigma^{-1}(U)$.   
    Hence $\sigma^{-1}(U)$ is open in $Y$,
    and then $\sigma$ is continuous on $Y$. 
    By a quite similar argument, we see that $\tau$, 
    the inverse map of $\sigma$, 
    is continuous on $X$.
    Thus we conclude that $\sigma$ is a homeomorphism.

    It remains to prove that $\a\colon Y\to\set{-1, 1}$ 
    is continuous on $Y$. 
    Let $q_0\in Y$. 
    There exists $h_{0}\in\sx$ such that $h_0(\sigma(q_0))=1$. 
    Since $h_0\circ \sigma$ is continuous, there exists an open neighborhood $W$ of $q_0$ such that $h_0\circ \sigma\neq 0$ on $W$. 
    Thus it follows from \eqref{EqL} that 
    \begin{equation*}
        \a =\frac{\Phi(h_0)}{h_0\circ \sigma} \q \text{on}\hspace{0.5em} W, 
    \end{equation*} 
    which implies that $\a$ is continuous on the open neighborhood $W$ of $q_0$. 
    Since $q_0$ is arbitrarily chosen, $\a$ is continuous on $Y$.  
    This completes the proof.
\end{sproof}

Finally, we determine the form of the surjective phase-isometry $T$. 
    Since $T \simeq \Phi$ by Lemma~\ref{Pmax}, we can find 
    $\varepsilon\colon\sx\to\set{-1, 1}$ such that  $T=\varepsilon\Phi$. 
It therefore follows from \eqref{EqL} that, for every $f \in \sx$ and $q \in Y$, 
    \begin{equation*}
        T(f)(q)=\varepsilon(f)\Phi(f)(q)=\varepsilon(f)\a(q) f(\sigma(q)). 
    \end{equation*}
The proof is complete. 
\end{proof}

\section{Remarks}
\begin{rem}
As we mentioned Remark~\ref{rem1}, we demonstrate that the real Banach space $C_0(L, \R)$ is a CL-space. 
The $\ell^\infty$-direct sum of $C_0(L, \R)$ and the one dimensional Banach space $\R$ is isometrically isomorphic to $C(L_\infty, \R)$: 
\begin{align*}
C(L_\infty,\R) \cong \cl \oplus_{\ell^\infty} \R,
\end{align*}
where $L_\infty$ stands for the one-point compactification of $L$. 
Furthermore, it is known that $C(L_{\infty},\R)$ is a CL-space (see \cite[Chapter~3]{Lima}). 
Thus the space $C_0(L, \R)$ is an $\ell^\infty$-summand of a CL-space. 
As a direct consequence of \cite[Proposition~8]{MiguelRafael}, we see that every $\ell^\infty$-summand of a CL-space is also a CL-space. 
This shows that $C_0(L, \R)$ is a CL-space. 
\end{rem}

While our main result can be regarded as a corollary of \cite[Theorem 3.4]{TZH}, a certain step in their proof do not seem to be fully established.

\begin{rem}
    In this remark, we use the same notation as in \cite{TZH}.
    In the proof of \cite[Theorem~3.4]{TZH}, the authors of \cite{TZH} attempt to construct a map $g$ from the unit sphere $S_X$ of a CL-space $X$ to the unit sphere $S_Y$ of an arbitrary Banach space $Y$. 
    Note that, by Lemma~\ref{lemHOM}, for every maximal convex subset $F$ of $S_X$, there exists an extreme point $x^*$ of the closed unit ball $B_{X^*}$ of the dual space $X^*$ of $X$ such that $F=\set{x\in S_X:x^*(x)=1}$.   
    The set on the right-hand side is denoted by $F_{x^*}$ in \cite{TZH}.
    Define the set $\mathrm{mex}(B_{X^*})$ as the set of all extreme points $x^*$ of $B_{X^*}$ such that $F_{x^*}$ is a maximal convex subset of $S_X$.
    By the axiom of choice, we define a subset $\mathrm{mex}^+(B_{X^*})$ of $\mathrm{mex}(B_{X^*})$ such that, for each $x^*\in \mathrm{mex}(B_{X^*})$,  exactly one of $x^*$ and $-x^*$ belongs to $\mathrm{mex}^+(B_{X^*})$. 
    We note that, every maximal convex subset $F$ of $S_X$ can be written as $F=tF_{x^*}=F_{tx^*}$ for some $t\in\set{-1,1}$ and $x^*\in \mathrm{mex}^+(B_{X^*})$. 
    In fact, for every maximal convex subset  $F$ of $S_X$, there exists $x^*\in \mathrm{mex}(B_{X^*})$ such that $F=F_{x^*}$. 
    By the definition of $\mathrm{mex}^+(B_{X^*})$, there exist $t\in\set{-1,1}$ and $z^*\in \mathrm{mex}^+(B_{X^*})$ such that $z^*=tx^*$. 
    Then $F=F_{tz^*}$.

    The map $g$ is constructed as follows:
    Fix a base point $x^*_0\in \mathrm{mex}^+(B_{X^*})$. 
    For every $f\in S_X$, we have $x\in F_{tx^*}$ for some $t\in\set{-1,1}$ and  $x^*\in \mathrm{mex}^+(B_{X^*})$. 
    By \cite[Lemma~3.3]{TZH}, there exist a sign $\theta_{tx^*}\in\set{-1,1}$ and a map $g_{tx^*}\colon S_X \to S_Y$ such that $g_{x_0^*}=\theta_{tx^*}g_{tx*}$ on $F_{x^*_0}\cap  F_{tx^*}$.
    Using $\theta_{tx^*}$ and $g_{tx^*}$, the authors of \cite{TZH} define a map $g$ by
    \begin{equation*}
        g(x)=\begin{cases}
            g_{x^*_0}(x)&\text{if}\ x\in F_{x_0^*},\\
            -g_{x^*_0}(-x)&\text{if}\ x\in -F_{x_0^*},\\
            \theta_{x^*}g_{x^*}(x)&\text{if}\ x\in F_{x^*},\\
            -\theta_{x^*}g_{x^*}(-x)&\text{if}\ x\in -F_{x^*}. 
        \end{cases}
    \end{equation*}
    To guarantee that the map $g$ is well-defined, we assume that an element $x\in S_X$ belongs to two maximal convex subsets $tF_{x^*}$ and $s F_{z^*}$ for some $t,s\in\set{-1,1}$ and $x^*,z^*\in \mathrm{mex}^+(B_{X^*})$. 
    In their proof, the case where $x\in F_{x^*}\cap F_{z^*}$ is proved.
    However, it appears to us that the remaining cases are not fully justified in the proof.
    For example, consider the case where $x\in F_{x^*}\cap -F_{z^*}=F_{x^*}\cap F_{-z^*}$.
    In this case, we need to show that $\theta_{x^*}g_{x^*}(x)=-\theta_{z^*}g_{z^*}(-x)$. 
    By the same argument as in the case $x\in F_{x^*}\cap F_{z^*}$, we can lead to $\theta_{x^*}g_{x^*}(x)=\theta_{-z^*}g_{-z^*}(x)$, but we are not sure that $\theta_{-z^*}g_{-z^*}(x)=-\theta_{z^*}g_{z^*}(-x)$ holds or not.
    In contrast, Lemma~\ref{lem3. 19} establishes  the well-definedness of $\Phi$, which plays the same role as $g$, by considering all possible cases.
\end{rem}

\subsection*{Acknowledgement}
The second author was supported by JST SPRING, Grant Number JPMJSP2121.
The authors would like to express our sincere gratitude to Professor Takeshi Miura, our supervisor, for his valuable guidance and support.
% The third author was partially supported by JSPS KAKENHI
% Grant Number JP 25K07028.

\end{document}